\documentclass[a4paper,10pt]{amsart}

\usepackage{amsmath, amssymb, amsfonts, amsthm}

\usepackage{amsfonts}

\usepackage{bm}

\usepackage{hyperref}

\hypersetup{
    colorlinks = true,
}

\usepackage[ansinew]{inputenc}

\def\F{\mathbf{F}}
\def\Z{\mathbf{Z}}

\def\Q{\mathbf{Q}}

\DeclareMathOperator{\Tr}{Tr}
\DeclareMathOperator{\Ker}{Ker}
\DeclareMathOperator{\Gal}{Gal}
\DeclareMathOperator{\pgcd}{gcd}

\DeclareMathOperator{\rad}{rad}

\theoremstyle{plain}
\newtheorem{theorem}{Theoreme}[section]

\newtheorem{proposition}[theorem]{Proposition}
\newtheorem{lemma}[theorem]{Lemma}
\newtheorem{corollary}[theorem]{Corollary}

\theoremstyle{definition}
\newtheorem{definition}[theorem]{Definition}
\theoremstyle{remark}
\newtheorem{remark}[theorem]{Remark}

\numberwithin{equation}{section}

\begin{document}

\title{Zeta functions of quadratic Artin-Schreier curves in characteristic two}

\begin{abstract}
The aim of this paper is twofold: on one hand we study the invariants of traces of quadratic forms over a finite field of characteristic two. On the other hand, we give results about the zeta functions of certain curves studied by van der Geer and ven der Vlugt.
\end{abstract}



\author{Régis Blache}
\address{LAMIA, Université des Antilles}
\email{regis.blache@univ-antilles.fr}
\author{Timothé Pierre}
\address{Université d'\'Etat d'Haïti}
\email{pierretimothe1979@yahoo.fr}

\date{\today}

\maketitle

\markleft{\sc Blache, Pierre}

\markright{\sc Zeta functions of quadratic Artin-Schreier curves}

\section{Introduction}

We denote by $\F_{2^m}$ a finite field of characteristic $2$, and by 
\[
R:=\sum_{i=0}^d a_ix^{2^i} \in \F_{2^m}[x], a_d\neq 0
\]
a $2$-linearized (or additive) polynomial. We also set $f(x)=xR(x)$.

The family of (non singular, projective) Artin-Schreier curves having an affine equation of the form
\[
C_R~:~y^2+y=xR(x)
\] 
is our main object of study. It was introduced in \cite{vdgvdv}. These curves have beautiful properties, such as being supersingular, or having a large group of automorphisms. Moreover, many examples of maximal curves are of this form \cite{cakoz}. Finally, they also have been used in \cite{vv2} to construct supersingular curves of any genus over a finite field of characteristic two. We call these curves \emph{quadratic} Artin-Schreier curves.

Their study also has numerous applications to information theory: in coding theory their numbers of rational points give the weight enumerators of some Reed-Muler codes, and they can also be used to construct certain binary sequences. 

Here we shall concentrate on their zeta functions. Recall that if $\# C_R(\F_{2^{mn}})$ denotes the number of rational points of the curve $C_R$ over the degree $n$ extension of the base field, its zeta function is defined by
\[
Z(C_R,T)=\exp\left(\sum_{n\geq 1} \# C_R(\F_{2^{mn}}) \frac{T^n}{n}\right)
\]
This is a rational function from Weil(s proof of the Riemann hypothesis for curves; more precisely, its denominator is $(1-T)(1-2^mT)$, and its numerator $L(C_R,T)$, called the $L$-function of the curve, is a polynomial of degree $2g=2^d$, where $g$ denotes the genus of the curve $C_R$.

On the other hand, we consider for any $n\geq 1$ the sum and the $L$-function
\[
S_n(f):=\sum_{x\in \F_{2^{mn}}} (-1)^{\Tr_{\F_{2^{mn}}/\F_2}\circ f(x)},~L(f,T):=\exp\left(\sum_{n\geq 1} S_n(f) \frac{T^n}{n}\right)
\]
It is well-known that we have $\# C(\F_{2^{mn}})=1+2^{mn}+S_n(f)$, which gives the equality
\[
L(C_R,T)=L(f,T)
\]
and the link between the two objects.

Since the function $x\mapsto \Tr_{\F_{2^{mn}}/\F_2}\circ f(x)$ is a quadratic form over the $\F_2$-vector space $\F_{2^{mn}}$, the exponential sum $S_n(f)$ is determined by the isometry class of this form, i.e. by the dimension of the radical of the associated bilinear form, and an invariant $\varepsilon_n(f)$. The radical depends on the solutions of the so-called kernel equation, and is in principle easy to compute once we know the decomposition field of this equation. The invariant $\varepsilon_n(f)$ is finer, and there have been many attempts to compute it in general (see for instance \cite{hou}) or in particular cases (see \cite{fitz} for binomials in $\F_2[x]$ or \cite{oss} for forms with a big radical).

Here we show that all the invariants $\varepsilon_n(f)$ depend on the finite number of those $\varepsilon_d(f)$ for which $d$ divides (twice) the degree of the decomposition field of the kernel polynomial. 

As a consequence, we give a factorization of the $L$-function $L(f,T)$: the factors are almost cyclotomic polynomials, and we express their multiplicities from the above data, namely the dimensions of the radicals, and the invariants. This improves on \cite[Theorems 10.1 and 10.2]{vdgvdv} where the zeta function is determined only over some particular base fields.

Let us describe our methods in a few words. Whereas the proofs in \cite{vdgvdv} are mostly geometric, we reason here in a much more arithmetic (and elementary) way.

The first observation, Proposition \ref{fpL}, is well-known (see \cite{my} for instance). Since the curves are supersingular, the reciprocal roots of the $L$-function are almost roots of unity, and its factors are almost cyclotomic polynomials. If we define the period as the least common multiple of the orders of these roots of unity, then there must be some periodicity in the number of points from the very definition of the zeta function. 

The second observation is Proposition \ref{percar2}. Since we have rather explicit evaluations of the exponential sums associated to quadratic functions, we can determine the period from the knowledge of at most two values of the invariant. 

Once we have observed these two facts, the results follow in a completely elementary way from the properties of some well known arithmetic functions. 

The paper is organized as follows. In section \ref{prel}, we recall (and prove when necessary) some technical results that we use later in the paper. Then in Section \ref{gen}, we determine the period from the degree of the decomposition field of the kernel polynomial and some invariants. In Section \ref{even} and \ref{odd}, we give the main results, respectively in the cases of even and odd $m$: we give the relations between the invariants, determine them in some cases, and express the multiplicities of the cyclotomic factors. Finally, we treat an example associated to the Suzuki curve in the last Section, in order to illustrate the preceding results.

\section{Preliminaries}
\label{prel}

\subsection{Factorization of cyclotomic polynomials}

\label{factcyc}

We first need to determine the factorization of cyclotomic polynomials over the field $\Q(\sqrt{2})$.

For any $n\geq 1$, we set $\zeta_n:=e^{\frac{2i\pi}{n}}$. 

Since we have $\sqrt{2}=\zeta_8+\zeta_8^7$, the field $\Q(\sqrt{2})$ is the subfield of $\Q(\zeta_8)$ fixed by the Galois automorphism defined by $\zeta_8\mapsto \zeta_8^7$. From this observation, we deduce that for $v_2(\ell)<3$, the fields $\Q(\zeta_\ell)$ and $\Q(\sqrt{2})$ are linearly disjoint, and the polynomial $\Phi_\ell(T)$ remains irreducible over $\Q(\sqrt{2})$.

If we have $v_2(\ell)\geq 3$, then $\Q(\sqrt{2})$ is the subfield of $\Q(\zeta_\ell)$ fixed by the subgroup $H$ of $\Gal(\Q(\zeta_\ell)/\Q)$ corresponding to
\[
\left\{k\in (\Z/\ell\Z)^\times,~k\equiv \pm 1\mod 8\right\}
\]
This is the kernel of the following character

\begin{definition}
We denote by $\chi$ the Dirichlet character of modulus $8$ defined by $\chi(3)=\chi(5)=-1$. 
\end{definition}

The action of the group $H$ on the set of primitive $\ell$-th roots of unity has two orbits, namely
\[
\mu_\ell^{\times+}=\{\zeta_\ell^i,~0\leq i\leq \ell-1,~\chi(i)=1\}, ~\mu_\ell^{\times-}=\{\zeta_\ell^i,~0\leq i\leq \ell-1,~\chi(i)=-1\}
\]
As a consequence, the factorization of $\Phi_\ell$ over $\Q(\sqrt{2})$ is 
\[
\Phi_\ell(T)=\Phi_\ell^+(T)\Phi^-_\ell(T),~\Phi_\ell^\pm(T):=\prod_{i,~\chi(i)=\pm 1} (1-\zeta_\ell^i T)=\prod_{\zeta\in \mu_\ell^{\times\pm}} (1-\zeta T)
\]

\subsection{Evaluation of certains sums of roots of unity}

We introduce two families of sums of roots of unity

\begin{definition}
\label{rama}
First, the \emph{Ramanujan sums} \cite{rama}: for any $\ell,n\geq 1$
\[
c_\ell(n):=\sum_{i\in (\Z/\ell\Z)^\times} \zeta_\ell^{ni}
\]
Second,  for any $\ell$ multiple of $8$ and $n$ the sums
\[
\sigma_\ell(n):=\sum_{i\in (\Z/\ell\Z)^\times} \chi(i)\zeta_\ell^{ni}
\]
\end{definition}

If $\varphi$ denotes Euler's totient, and $\mu$ the Möbius function, the Ramanujan sums have the following well known expression, called the von Sterneck arithmetic function 
\begin{equation}
\label{vs}
c_\ell(n)=\mu\left(\frac{\ell}{\pgcd(\ell,n)}\right)\frac{\varphi(\ell)}{\varphi\left(\frac{\ell}{\pgcd(\ell,n)}\right)}
\end{equation}

For the second family of sums, we have the following

\begin{lemma}
\label{sigma}
Write $\ell=2^k\ell'$, $k=v_2(\ell)\geq 3$; then we have 
\[
\sigma_\ell(n)=\left\{\begin{array}{rcl} \chi(\ell'n')2^{k-2}\sqrt{2}c_{\ell'}(n') & \textrm{if} & n=2^{k-3}n',~n'{\rm odd}\\
0 & \textrm{if} & v_2(n)\neq v_2(\ell)-3\end{array}\right.
\]
\end{lemma}

\begin{proof}
We first write Bezout identity $2^ku+\ell'v=1$; from the Chinese remainder theorem, we deduce that can rewrite the sum (recall that $\chi$ is defined modulo $8$)
\begin{eqnarray*}
\sigma_\ell(n) & = & \sum_{a\in (\Z/2^k\Z)^\times}\sum_{b\in (\Z/\ell'\Z)^\times} \chi(2^kub+\ell'va)\zeta_\ell^{n(2^kub+\ell'va)}\\
& = & \chi(\ell') \sum_{a\in (\Z/2^k\Z)^\times}\chi(va)\zeta_{2^k}^{nva}\sum_{b\in (\Z/\ell'\Z)^\times} \zeta_{\ell'}^{nub}\\
& = & \chi(\ell') \sum_{a\in (\Z/2^k\Z)^\times}\chi(a)\zeta_{2^k}^{na}\sum_{b\in (\Z/\ell'\Z)^\times} \zeta_{\ell'}^{nub}\\
\end{eqnarray*}
We recognize that the last sum is the Ramanujan sum $c_{\ell'}(nu)=c_{\ell'}(n)$.

If we write $n=2^tn'$, with $n$ odd, we get that the sum over $(\Z/2^k\Z)^\times$ is equal to 
\begin{itemize}
	\item[(a)] $\sum_{a\in (\Z/2^k\Z)^\times}\chi(a)=0$ if  $t\geq k$;
	\item[(b)] $-\sum_{a\in (\Z/2^k\Z)^\times}\chi(a)=0$ if  $t= k-1$;
	\item[(c)] $\sum_{a\in (\Z/2^k\Z)^\times}\chi(a)i^{n'a}=0$ if  $t= k-2$.
\end{itemize}

If $t\leq k-3$, we set $a=a_0+8a_1$, $a_0\in(\Z/8\Z)^\times$, $a_1\in \Z/2^{k-3}\Z$; then
\[
\sum_{a\in (\Z/2^k\Z)^\times}\chi(a)\zeta_{2^k}^{na}=\sum_{a_0}\chi(a_0)\zeta_{2^{k-t}}^{n'a_0}\sum_{a_1}\zeta_{2^{k-t-3}}^{n'a_1}
\]
The last sum is zero, unless we have $t=k-3$ and then it is equal to $2^{k-3}$. The sum over $a_0$ is equal to $\chi(n')2\sqrt{2}$, and this gives the result.
\end{proof}

\subsection{Some matrices whose entries are arithmetic functions}
\label{evalsum}

We introduce here two sequences of matrices for future use

\begin{definition}
For any integer $n\geq 1$, we set
\[
A(n):=\left(c_\ell(d)\right)_{d,\ell|n},~B(n):=\left(\sigma_\ell(d)\right)_{d,\ell|n}
\]
\end{definition}

The matrix $A(n)$ is invertible: in order to see this, it is sufficient to slightly modify the argument in the proof of \cite[Theorem 9]{apo} to verify that its determinant is the product of the divisors of $n$.

Set $n=2^an'$, with $n'$ odd. Then we can write the matrix $A(n)$ in the following block form
\[
A(n)=(A(n)_{ij})_{0\leq i, j \leq a},~A(n)_{ij}:= \left(c_\ell(d)\right)_{d,\ell|N,~v_2(d)=i,~v_2(\ell)=j}
\]

Using von Sterneck arithmetic function, we have
\begin{equation}
\label{matrixA}
A_{ij}(n) = \left\{ \begin{array}{rcl}
0 & \textrm{if} & i\leq j-2\\
-2^iA(n') & \textrm{if} & i= j-1\\
A(n') & \textrm{if} & j=0\\
2^{j-1}A(n') & \textrm{if} & i\geq j\geq 1\\
\end{array}\right.
\end{equation}

We turn our attention to the matrix $B(n)$. We introduce a diagonal matrix

\begin{definition}
Let $n'$ denote an odd integer. The matrix $\Delta(n')$ is the diagonal matrix whose coefficients are the $\chi(\ell)$, $\ell|n'$.
\end{definition}

We can write it in block form as above, with blocks $B_{ij}(n)$. From the expression in Lemma \ref{sigma}, we see that all blocks are zero, except the blocks $B(n)_{ii+3}$ and that 
\begin{equation}
\label{matrixB}
B(n)_{ii+3}=2^{i-2}\sqrt{2}\Delta(n') A(n')\Delta(n')
\end{equation}

\subsection{Quadratic forms over a finite field of characteristic two}

We recall the classification, up to isometry, of quadratic forms over a finite dimensional $\F_2$-vector space.

Let $q : V \rightarrow \F_2$ denote a quadratic form over a $\F_2$-vector space $V$ of dimension $k$. We associate to $q$ a bilinear form, its polarisation, defined by $b(x,y)=q(x+y)+q(x)+q(y)$; then $b$ is alternate, and it does not depend on the diagonal part of $q$. We no longer have a one to one correspondance between the quadratic forms and the bilinear forms over $V$.

We denote by $\rad b:=V^\perp$ the radical of $b$, and by $c$ its dimension (which has the same parity as $k$). Then there exists a basis $(e_1,\ldots,e_r,e_{r+1},\ldots,e_k)$ of $V$, such that $(e_{r+1},\ldots,e_k)$ is a basis of $\rad b$ and $q$ can be written in (the dual basis of) this basis in one of the following ways \cite[Theorem 6.30]{ln}
\begin{itemize}
	\item[(i)] $q(x)=x_1x_2+\cdots+x_{r-1}x_r+x_{r+1}^2$ 
	\item[(ii)] $q(x)=x_1x_2+\cdots+x_{r-1}x_r$
	\item[(iii)] $q(x)=x_1^2+x_1x_2+x_2^2+\cdots+x_{r-1}x_r$
\end{itemize}

\begin{remark}
\label{zeroinv}
Remark that the first case corresponds to the quadratic forms that are not trivial on the radical of their polarisation, and the last two to the forms having trivial restriction. 
\end{remark}

We now define the invariants that we shall study.

\begin{definition}
To each isometry class, we associate an invariant which is respectively $0,1,-1$ in each of the cases (i), (ii) or (iii) and that we denote by $\varepsilon(q)$. 
\end{definition}

We have the following \cite[Theorem 6.32]{ln}
\begin{proposition}
\label{expquad}
The exponential sum associated to the quadratic form $q$ satisfies
\[
\sum_{x\in V} (-1)^{q(x)}=\varepsilon(q)2^{\frac{k+c}{2}}
\]
\end{proposition}

\section{General results on the zeta functions}
\label{gen}

We fix once and for all a $2$-linear polynomial of degree $2^d$, $R:=\sum_{i=0}^d a_ix^{2^i}$ in $\F_{2^m}[x]$, and we set $f(x):=xR(x)$.

We consider the non singular projective curve $C_R$ defined over $\F_{2^m}$ by the affine equation $y^2+y=f(x)$. This is an hyperelliptic curve (equivalently, an Artin-Schreier covering of the projective line, since the characteristic is two) with genus $g=2^{d-1}$. Moreover it is supersingular \cite[Theorem 9.4]{vdgvdv}.

Since the point at infinity of the projective line is totally ramified in the covering, the number of rational points of this curve over the field $\F_{2^{mn}}$ is
\[
\# C_R(\F_{2^{mn}})=1+2^{mn}+\sum_{x\in \F_{2^{mn}}} (-1)^{\Tr_{\F_{2^{mn}}/\F_2}\circ f(x)}=1+2^{mn}+S_n(f)
\]
and the numerator of the zeta function $Z(C_R,T)$ is the $L$-function $L(f,T)$.

In the following, we focus on this last function.

\subsection{First properties of the quadratic forms}

Let us first define our main objects of study

\begin{definition}
For each integer $n\geq 1$, we denote by $q_n$ the quadratic form $q_n:=\Tr_{\F_{2^{mn}}/\F_2}\circ f$ from $\F_{2^{mn}}$ to $\F_2$. We denote by $b_n$ its polarisation, and by $\rad(b_n)$ its radical.

We denote respectively by $c_n(f)$ and $\varepsilon_n(f)$ the codimension of the radical of $b_n$, and its invariant. 
\end{definition}

From Proposition \ref{expquad}, we have
\[
S_n(f)=\varepsilon_n(f)2^{\frac{mn+c_n(f)}{2}}
\]

These forms, and the associated sums, have already been studied in many papers; let us just cite \cite{fitz,hou}. We extract some results from \cite{hou}, that we reprove for completeness.

First about the radical $\rad(b_n)$. To the additive polynomial $R$, we associate another additive polynomial

\begin{definition}
\label{decomp}
The \emph{kernel polynomial} associated to the family of quadratic forms $(q_n)$ is the polynomial
\[
\widetilde{R}:=(R+R^\ast)^{2^d}=\sum_{i=0}^d a_i^{2^d}x^{2^{d+i}}+a_i^{2^{d-i}}x^{2^{d-i}}
\]
where $R^\ast$ is the adjoint of the polynomial $R$.

We denote by $\F_{2^{mN}}$ the decomposition field of $\widetilde{R}$ over $\F_{2^m}$, and by $\Ker\widetilde{R}\subset \F_{2^{mN}}$ the set of roots of this polynomial.

In the following, we set $N=2^aN'$, with $N'$ odd.
\end{definition}

Fix an $n\geq 1$; for any $x,y \in \F_{2^{mn}}$, we have 
\[
b_n(x,y)=\Tr_{\F_{2^{mn}}/\F_2}(xR(y)+R(x)y)=\Tr_{\F_{2^{mn}}/\F_2}(x^{2^d}\widetilde{R}(y))
\] 

As a consequence, since the bilinear form $(x,y)\mapsto \Tr_{\F_{2^{mn}}/\F_2}(xy)$ is non degenerate, we have 
\[
\rad(b_n)=\Ker\widetilde{R}\cap \F_{2^{mn}}
\] 
Since the degree $1$ coefficient of $\widetilde{R}$ is $a_d\neq 0$, this polynomial is separable, all its roots are simple, and we have $2^{c_n(f)}=\deg\pgcd(\widetilde{R},x^{2^{mn}}+x)$.

Note that $\widetilde{R}$ divides $x^{2^{mN}}+x$ and we have $c_N(f)=2d$.

We begin with 

\begin{lemma} \cite[Propositions 3.1 and 3.3]{hou}
\label{houhou}
Notations are as above
\begin{itemize}
	\item[1.] for any $n\geq 1$, we have $c_n(f)=c_{\pgcd(n,N)}(f)$;
	\item[2.] if moreover $v_2(n)>v_2(N)$, then $\varepsilon_n(f)\neq 0$.
\end{itemize}
\end{lemma} 

\begin{proof}
Let $n\geq 1$; then we have
\begin{eqnarray*}
\pgcd(\widetilde{R},x^{2^{mn}}+x) & = &\pgcd(\pgcd(\widetilde{R},x^{2^{mN}}+x),x^{2^{mn}}+x))\\
& = & \pgcd(\widetilde{R},\pgcd(x^{2^{mN}}+x,x^{2^{mn}}+x))\\
& = & \pgcd(\widetilde{R},x^{2^{m\pgcd(n,N)}}+x)\\
\end{eqnarray*}
From the equality $2^{c_n(f)}=\deg\pgcd(\widetilde{R},x^{2^{mn}}+x)$, we deduce the first assertion.

Assume that $v_2(n)>v_2(N)$; then we have $n=d\pgcd(n,N)$ for some even $d$. We have seen that $\rad(b_n)=\rad(b_{\pgcd(n,N)})$; if $x$ lies in this subspace, $f(x)$ is in $\F_{2^{m\pgcd(n,N)}}$, and
\begin{eqnarray*}
q_n(x) & = & \Tr_{\F_{2^{mn}}/\F_2}(f(x))\\
 & = & \Tr_{\F_{2^{m\pgcd(n,N)}}/\F_2}\left(\Tr_{\F_{2^{mn}}/\F_{2^{m\pgcd(n,N)}}}(f(x))\right)\\
& = & \Tr_{\F_{2^{m\pgcd(n,N)}}/\F_2}(df(x))=0 \\
\end{eqnarray*}

We deduce that the restriction of $q_n$ to $\rad(b_n)$ is trivial, and the second assertion from Remark \ref{zeroinv}.

\end{proof}

\subsection{First properties of the {\it L}-function}

We shall study a new function, close to the $L$-function, but with simpler arithmetical properties

\begin{definition}
The modified $L$-function is 
\[
L^\ast(f,T):=L\left(f,\frac{T}{\sqrt{2}^m}\right)
\]
\end{definition}

\begin{remark}
In the same way as the $L$-function comes from the sums $(S_n(f))_{n\geq 1}$, the modified $L$-function comes from the modified sums
\begin{equation}
\label{modsum}
S_n^\ast(f)=(\sqrt{2})^{-mn}S_n(f)=\varepsilon_n(f)2^{\frac{c_n(f)}{2}}=\varepsilon_n(f)2^{\frac{c_{\pgcd(n,N)}(f)}{2}}
\end{equation}

from the first part of Lemma \ref{houhou}.
\end{remark}

We list the first properties of this new function in the following

\begin{proposition}
\label{fpL}
The function $L^\ast(f,T)$ satisfies 
\begin{itemize}
	\item[(i)]  it is a polynomial of degree $2^d$, with coefficients in $\Z[\sqrt{2}]$;
  \item[(ii)] its reciprocal roots are roots of unity.
	\item[(iii)] if $m$ is even, then it has integer coefficients.
\end{itemize}
\end{proposition}

\begin{proof}
We only show assertion (ii): the other assertions follow readily from the fact that the $L$ function $L(f,T)$ is a polynomial  of degree $2^d$ with integer coefficients. 

The reciprocal roots of the modified $L$-function are the $\beta_i=\alpha_i/\sqrt{2}^m$, $1\leq i\leq 2^d$, where the $\alpha_i$ are the reciprocal roots of the function $L(f,T)$. For any odd prime $\ell$, these numbers are $\ell$-adic units from Weil's proof of the Riemann hypothesis over finite fields.

We consider their $2$-adic valuations. Since the curve is supersingular, we have $v_2(\alpha_i)=m/2$ for all $i$, and we deduce that the $\beta_i$ are $2$-adic units. Thus all the $\beta_i$ are algebraic integers.

Finally, since all conjugates of the $\beta_i$ have complex module $1$, a classical theorem of Kronecker \cite{kro} ensures that they are roots of unity.
\end{proof}

We borrow the following definition to \cite{my}

\begin{definition}
The \emph{period} $D$ of the function $L(f,T)$ is the least common multiple of the orders of the reciprocal roots of the modified $L$ function.
\end{definition}

The period has a simple expression in the degree of the decomposition field of the polynomial $\widetilde{R}$

\begin{proposition}
\label{percar2}
Recall that $\F_{2^{mN}}$ is the decomposition field of $\widetilde{R}$ over $\F_{2^m}$. Then the period satisfies $D\in\{N,2N,4N\}$.

Precisely, we have the following cases
\begin{itemize}
	\item[(i)] if $\varepsilon_N(f)=-1$, then $D=N$;
	\item[(ii)] if $\varepsilon_N(f)=1$, then $\varepsilon_{2N}(f)=-1$, $D=2N$, and all roots orders have dyadic valuation $v_2(N)+1$;
	\item[(iii)] if $\varepsilon_N(f)=0$, we have the following alternative
	
	\begin{itemize}
		\item[(iiia)] if $\varepsilon_{2N}(f)=-1$, then $D=2N$;
		\item[(iiib)] if $\varepsilon_{2N}(f)=1$, then $D=4N$, and all roots orders have dyadic valuation $v_2(N)+2$.
	\end{itemize}
\end{itemize}
\end{proposition}

\begin{proof}
If we compare the logarithmic derivatives of both sides of the equality
\[
L^\ast(f,T)=\prod_{i=1}^{2^d} (1-\beta_iT)
\]
we see that for any $n\geq 1$ the modified sum can be written from the reciprocal roots of the modified $L$-function as 
\begin{equation}
\label{modsums}
S_n^\ast(f)=-\sum_{i=1}^{2^d} \beta_i^n
\end{equation} 

From the expression (\ref{modsum}) of the sum $S_n^\ast(f)$, and since we have $c_N(f)=2d$ by definition of the decomposition field, we deduce
\[
-\sum_{i=1}^{2^d} \beta_i^N=\varepsilon_N(f)2^d
\]
In the case $\varepsilon_N(f)=-1$, the triangle inequality ensures that $\beta_i^N=1$ for all $i$. Thus $D$ divides $N$. Since $N$ is the least integer with $c_N(f)=2d$, we get assertion (i).

When $\varepsilon_N(f)=1$, we get $\beta_i^N=-1$ and $\beta_i^{2N}=1$ for all $i$. Thus $D$ divides $2N$, and the root orders all have dyadic valuation equal to $v_2(N)+1$. Here again, $N$ is the least integer such that $\left|\sum_{i=1}^{2^d} \beta_i^N\right|=2^d$, and we have $D=2N$.

In the case $\varepsilon_N(f)=0$, we have $\varepsilon_{2N}(f)=\pm 1$ from Lemma \ref{houhou} (ii). Then we conclude as above from the value of $\varepsilon_{2N}(f)$ since we have $c_{2N}(f)=2d$.
\end{proof}

\section{The case of even {\it m}}
\label{even}

In this section, $m$ is even.

In this case, the $\F_2$-vector space $\F_{2^{mn}}$ has even dimension for all $n$, and the corank $c_n(f)$ is even. Thus the modified $L$-function has integer coefficients; it is a product of cyclotomic polynomials from Proposition \ref{fpL} (2), and we write
\begin{equation}
\label{prodcyceven}
L^\ast(f,T)=\prod_{\ell} \Phi_\ell(T)^{m_\ell(f)}
\end{equation}
As a consequence, (\ref{modsums}) gives the following expression for the modified sums, where $c_\ell(n)$ is the Ramanujan sum from Definition \ref{rama}
\begin{equation}
\label{expeven}
-S_n^\ast(f)=\sum_{\ell} m_\ell(f) c_\ell(n).
\end{equation}

From Proposition \ref{percar2}, we deduce
\begin{proposition}
\label{meven}
Assume $m$ is even. The invariants $\varepsilon_n(f)$, $n\geq 1$, satisfy
\begin{itemize}
	\item[(i)] if $\varepsilon_N(f)=-1$, then $\varepsilon_n(f)=\varepsilon_{\pgcd(n,N)}(f)$;
	\item[(ii)] when $\varepsilon_N(f)=1$, we have
	
	\begin{itemize}
		\item if $v_2(n)\leq v_2(N)-1$, then $\varepsilon_n(f)=0$;
		\item if $v_2(n)= v_2(N)$, then $\varepsilon_n(f)=\varepsilon_{\pgcd(n,N)}(f)$;
		\item if $v_2(n)\geq v_2(N)+1$, then $\varepsilon_n(f)=-\varepsilon_{\pgcd(n,N)}(f)$;
	\end{itemize}
	
		\item[(iii)] when $\varepsilon_N(f)=0$, we have the following cases

		\item[(iiia)] if $\varepsilon_{2N}(f)=-1$, then $\varepsilon_n(f)=\varepsilon_{\pgcd(n,2N)}(f)$ for all $n\geq 1$.
		
	\item[(iiib)] if $\varepsilon_{2N}(f)=1$, then 
	
	\begin{itemize}
		\item if $v_2(n)\leq v_2(N)$, then $\varepsilon_n(f)=0$;
		\item if $v_2(n)= v_2(N)+1$, then $\varepsilon_n(f)=\varepsilon_{\pgcd(n,2N)}(f)$;
		\item if $v_2(n)\geq v_2(N)+2$, then $\varepsilon_n(f)=-\varepsilon_{\pgcd(n,2N)}(f)$;
	\end{itemize}

\end{itemize}
\end{proposition}

\begin{proof}

This is a consequence of proposition \ref{percar2}, equation (\ref{expeven}) and of the von Sterneck expression for Ramanujan sums (\ref{vs}).

First, when $\ell$ divides $D$, we have $c_\ell(n)=c_{\ell}(\pgcd(D,n))$, that ensures $\varepsilon_n(f)=\varepsilon_{\pgcd(n,D)}(f)$ for all $n\geq 1$. This proves assertions (i) and (iiia).

In case (ii), we have $v_2(\ell)=v_2(N)+1$ for any $\ell$ such that $m_\ell(f)\neq 0$ from Proposition \ref{percar2} (ii). This implies the following equalities
\begin{itemize}
	\item $c_\ell(n)=0$ if $v_2(n)\leq v_2(N)-1$,
	\item  $c_\ell(n)=c_\ell(\pgcd(n,2N))=c_\ell(\pgcd(n,N))$ if $v_2(n)= v_2(N)$, and
	\item  $c_\ell(n)=c_\ell(\pgcd(n,2N))=c_\ell(2\pgcd(n,N))=-c_\ell(\pgcd(n,N))$ else.
\end{itemize}

Case (iiib) is treated as case (ii), noting that $v_2(\ell)=v_2(N)+2$ for all $\ell$ such that $m_\ell(f)\neq 0$.
\end{proof}

\begin{remark}
Let us denote by $\sigma$ the number of divisors function, and set $N=2^aN'$, with $N'$ odd.

We deduce from the preceding result that the knowledge of the family $(S_n(f))_{n\geq 1}$ can be reduced to the knowledge of $\sigma(N)$ of these sums in case (i), $\sigma(2N)$ of these sums in case (iiia) and $\sigma(N')$ in the remaining cases.

These results are in the spirit of \cite[Theorem 1]{my}.
\end{remark}

We end this section with an expression for the multiplicities $m_\ell(f)$ in (\ref{prodcyceven}). 

\begin{proposition}
\label{multeven}
Assume $m$ is even; recall that we have set $N=2^a N'$, $N'$ odd. The multiplicities $m_\ell(f)$ satisfy the following systems, depending on the case from the above Proposition
\begin{itemize}
	\item[(i)] $A(N)(m_\ell(f))_{\ell|N}=(-S_d^\ast(f))_{d|N}$;
	\item[(ii)] $2^aA(N')(m_{2^{a+1}\ell}(f))_{\ell|N'}=(S_{2^{a}d}^\ast(f))_{d|N'}$;
	\item[(iiia)] $A(2N)(m_\ell(f))_{\ell|2N}=(-S_d^\ast(f))_{d|2N}$;
	\item[(iiib)] $2^{a+1}A(N')(m_{2^{a+2}\ell}(f))_{\ell|N'}=(S_{2^{a+1}d}^\ast(f))_{d|N'}$
\end{itemize}
\end{proposition}

\begin{proof}
We start with the system consisting of the equations (\ref{expeven}) for $n\geq 1$.

In case (i) (\emph{resp.} (iiia)) of the preceding proposition, this system is equivalent to the system consisting of the same equations, when $n$ runs over the divisors of $N$ (\emph{resp.} $2N$).
Now the corresponding assertions are just the matrix forms of this last system.

In case (ii), we first reduce to the system consisting of the same equations, when $n$ runs over the divisors of $N$. Then we use the block form of the matrix $A$. Since the only non zero multiplicities are the ones with $v_2(\ell)=a+1$, and the only non zero sums $S_n^\ast(f)$ are those with $v_2(n)=a$, the block corresponding to the remaining part of the system is $A_{aa+1}$, and we get the result from its description (\ref{matrixA}).

Case (iiib) is proven the same way, replacing $a$ by $a+1$.
\end{proof}

\section{The case of odd {\it m}}
\label{odd}

We first remark that when $m$ is odd, the equality $c_N(f)=2d$, joint to the fact that the rank  of a quadratic form is an even integer, force $mN$, and $N$ to be even.

We thus write as above $N=2^aN'$, with $N'$ odd and $a\geq 1$.

If $m$ is odd, the modified sums $S_n^\ast(f)$ are no longer integers, but algebraic integers in $\Z[\sqrt{2}]$. The same is true for the coefficients of the modified $L$-function; since its reciprocal roots are roots of unity, we get from \ref{factcyc} a factorization of the form

\[
L^\ast(f,T)=\prod_{\ell,~v_2(\ell)\leq 2} \Phi_\ell(T)^{m_\ell(f)}\prod_{\ell,~v_2(\ell)\geq 3} \Phi_\ell^+(T)^{m_\ell^+(f)} \Phi_\ell^-(T)^{m_\ell^-(f)}
\]
If we take logarithmic derivatives of both sides, we obtain the following expression for the modified sums
\[
-S_n^\ast(f):=\sum_{\ell,~v_2(\ell)\leq 2} m_\ell(f) c_\ell(n)+\sum_{\ell,~v_2(\ell)\geq 3} \left(m_\ell^+(f) c_\ell^+(n)+m_\ell^-(f) c_\ell^-(n)\right)
\]
where we have used the Ramanujan sums, and we have set
\[
c_\ell^\pm(n)=\sum_{i,~\chi(i)=\pm 1} \zeta_\ell^{ni},
\]

We now change the variables: we modify the multiplicities in order to make the sums from Definition \ref{rama} appear.

\begin{definition}
We define the \emph{positive multiplicities} associated to $f$ as
\[
M_\ell^+(f):=\left\{\begin{array}{rcl} m_\ell(f) & \textrm{if} & v_2(\ell)\leq 2 \\ \frac{m_\ell^+(f)+m_\ell^-(f)}{2} & \textrm{if} & v_2(\ell)\geq 3\\ \end{array}\right.
\]
and the \emph{negative multiplicities} associated to $f$ as
\[
M_\ell^-(f):=\left\{\begin{array}{rcl} 0 & \textrm{if} & v_2(\ell)\leq 2 \\ \frac{m_\ell^+(f)-m_\ell^-(f)}{2} & \textrm{if} & v_2(\ell)\geq 3\\ \end{array}\right.
\]
\end{definition}

Since we have $c_\ell^+(n)+c_\ell^-(n)=c_\ell(n)$ and $\sigma_\ell(n)=c_\ell^+(n)- c_\ell^-(n)$ we rewrite the modified sums as
\begin{equation}
\label{expodd}
-S_n^\ast(f):=\sum_{\ell} M_\ell^+(f) c_\ell(n)+\sum_{\ell,~v_2(\ell)\geq 3} M_\ell^-(f)\sigma_\ell(n)
\end{equation}

The rank of a bilinear form is even, and we have $c_n(f)\equiv n \mod 2$. We deduce that the modified sum $S_n^\ast(f)$ is in $\Z$ if $n$ is even, and in $\sqrt{2}\Z$ if $n$ is odd. As a consequence of the evaluations of the sums $c_\ell(n)$ and $\sigma_\ell(n)$ in \ref{evalsum}, we deduce
\begin{equation}
\label{expodd1}
\sum_{\ell} M_\ell^+(f) c_\ell(n)=\left\{\begin{array}{rcl} -S_n^\ast(f) & \textrm{if} & v_2(n)\geq 1 \\ 0 & \textrm{if} & v_2(n)=0\\ \end{array}\right.
\end{equation}

\begin{equation}
\label{expodd2}
\sum_{\ell,~v_2(\ell)\geq 3} M_\ell^-(f)\sigma_\ell(n)=\left\{\begin{array}{rcl} -S_n^\ast(f) & \textrm{if} & v_2(n)=0 \\ 0 & \textrm{si} & v_2(n)\geq 1\\ \end{array}\right.
\end{equation}

The first system is similar to the one given when $m$ is even, and we deduce the following equivalent of Proposition \ref{meven}

\begin{proposition}
\label{moddneven}
Assume $m$ is odd. For any even $n\geq 2$, the invariant $\varepsilon_n(f)$ satisfies
\begin{itemize}
	\item[(i)] if $\varepsilon_N(f)=-1$, then $\varepsilon_n(f)=\varepsilon_{\pgcd(n,N)}(f)$;
	\item[(ii)] if $\varepsilon_N(f)=1$, then 
	
	\begin{itemize}
		\item if $1\leq v_2(n)\leq v_2(N)-1$, then $\varepsilon_n(f)=0$;
		\item if $v_2(n)= v_2(N)$, then $\varepsilon_n(f)=\varepsilon_{\pgcd(n,N)}(f)$;
		\item if $v_2(n)\geq v_2(N)+1$, then $\varepsilon_n(f)=-\varepsilon_{\pgcd(n,N)}(f)$;
	\end{itemize}
	
		\item[(iii)] if $\varepsilon_N(f)=0$, then
		\item[(iiia)] when $\varepsilon_{2N}(f)=-1$, we have $\varepsilon_n(f)=\varepsilon_{\pgcd(n,2N)}(f)$;
		
	\item[(iiib)] when $\varepsilon_{2N}(f)=1$, we have
	
	\begin{itemize}
		\item if $1\leq v_2(n)\leq v_2(N)$, then $\varepsilon_n(f)=0$;
		\item if $v_2(n)= v_2(N)+1$, then $\varepsilon_n(f)=\varepsilon_{\pgcd(n,2N)}(f)$;
		\item if $v_2(n)\geq v_2(N)+2$, then $\varepsilon_n(f)=-\varepsilon_{\pgcd(n,2N)}(f)$.
	\end{itemize}

\end{itemize}
\end{proposition}

On the other hand, we also deduce expressions for the multiplicities as in Proposition \ref{multeven}

\begin{proposition}
\label{multodd+}
Assume $m$ is odd; recall that we have set $N=2^a N'$, $N'$ odd. The multiplicities $M_\ell^+(f)$ satisfy the following systems, depending on the case from the above Proposition
\begin{itemize}
	\item[(i)] $A(N)(M_\ell^+(f))_{\ell|N}=(-S_d^\ast(f))_{d|N}$;
	\item[(ii)] $2^aA(N')(M_{2^{a+1}\ell}^+(f))_{\ell|N'}=(S_{2^{a}d}^\ast(f))_{d|N'}$;
	\item[(iiia)] $A(2N)(M_\ell^+(f))_{\ell|2N}=(-S_d^\ast(f))_{d|2N}$;
	\item[(iiib)] $2^{a+1}A(N')(M_{2^{a+2}\ell}^+(f))_{\ell|N'}=(S_{2^{a+1}d}^\ast(f))_{d|N'}$
\end{itemize}
\end{proposition}

We now exploit the second system (\ref{expodd2}). 

We know from Proposition \ref{percar2} that all roots orders $\ell$ divide the period $D$, and $D\in\{N,2N,4N\}$. 

Since the sum $\sigma_\ell(n)$ is zero when $v_2(n)\neq v_2(\ell)-3$, the system (\ref{expodd2}) boils down to the $v_2(D)-2$ following systems
\begin{eqnarray*}
\sum_{\ell|D,~v_2(\ell)= 3} M_\ell^-(f)\sigma_\ell(n)=-S_n^\ast(f),& &v_2(n)=0 \\
\sum_{\ell|D,~v_2(\ell)= i} M_\ell^-(f)\sigma_\ell(n)=0, & & v_2(n)=i-3,~ 4\leq i\leq v_2(D) \\
\end{eqnarray*}
Recall the expression of $\sigma_\ell(n)$ from Lemma \ref{sigma}. 

First assume that $v_2(n)\neq v_2(\ell)-3$. In this case, since $\ell|D$, we have $v_2(\ell)\neq v_2(D)$, and $v_2(\pgcd(n,D))=\min\{v_2(n),v_2(D)\} \neq v_2(\ell)-3$. Thus we have 
\[
\sigma_\ell(\pgcd(n,D))=\sigma_\ell(n)=0
\]

When we have $v_2(n)= v_2(\ell)-3$, we also have $v_2(\pgcd(n,D))=v_2(\ell)-3$ since $\ell$ divides $D$. Moreover, we have $\pgcd(n,D)=2^{v_2(n)}\pgcd(n,N')$, and we deduce from Lemma \ref{sigma} that
\[
\sigma_\ell(\pgcd(n,D))=\chi(n\pgcd(n,N'))\sigma_\ell(n)
\]

When $n$ is odd, we deduce the following relation for the $\varepsilon_n(f)$
\begin{proposition}
\label{moddnodd}
Assume $m$ is odd; recall that we have set $N=2^a N'$, $N'$ odd. For any odd integer $n$, we have the equality
\[
\varepsilon_n(f)=\chi(n\pgcd(n,N'))\varepsilon_{\pgcd(n,N')}(f)
\]
\end{proposition}

On the other hand, if we set $\ell=2^i\ell'$, $\ell'|N'$, we deduce the following rewriting for the $v_2(D)-2$ systems
\begin{eqnarray*}
\sum_{\ell'|N'} M_{8\ell'}^-(f)\sigma_{8\ell'}(n)= -S_n^\ast(f), & & ~n|N'\\
\sum_{\ell'|N'} M_{2^i\ell'}^-(f)\sigma_{2^i\ell'}(2^{i-3}n)= 0, & & ~n|N',~ 4\leq i\leq v_2(D) \\
\end{eqnarray*}

The matrices associated to these systems are the $B(D)_{i-3,i}$ for $3\leq i \leq v_2(D)$  from (\ref{matrixB}), which are invertible.

This proves the following for the negative multiplicities

\begin{proposition}
\label{expmodd}
Recall that we have set $N=2^a N'$. Notations are as above. Then we have
\begin{itemize}
	\item[(1)] for any $\ell$ such that $v_2(\ell)\neq 3$, we have $M_\ell^-(f)=0$.
	\item[(2)] the non zero multiplicities $M_{8\ell'}^-(f)$, $\ell'|N'$ satisfy the system
	\[
	2\sqrt{2}\Delta(N') A(N')\Delta(N')\left(M_{8\ell'}^-(f)\right)_{\ell'|N'}=\left(S_n^\ast(f)\right)_{n|N'}
	\]
\end{itemize}
\end{proposition}

When the period is not divisible by $8$, or when the roots orders have dyadic valuation different from $3$, we see that the negative multiplicities are all equal to zero, and we deduce the following cancellations

\begin{corollary}
Assume $m$ is odd. For any odd $n\geq 1$, the sums $S_n(f)$ are zero and the modified $L$-function has integer coefficients when
\begin{itemize}
	\item[(a)] we have $v_2(D)\leq 2$;
	\item[(b)] we are in case (ii) and $v_2(N)=v_2(D)-1\neq 2$;
	\item[(c)] we are in case (iiib) and $v_2(N)=v_2(D)-2\neq 1$;	
\end{itemize}
where the cases are those of Proposition \ref{percar2}.
\end{corollary}

\section{An example and an application around the Suzuki curve}
\label{appli}

In this section, we fix an integer $h\geq 1$, and we set $q_0:=2^h$, $q:=2^{2h+1}=2q_0^2$.

We consider the polynomial $f(x)=x^{q_0}(x^q+x)$ in the following, and we determine all sums $S_n(f)$, $n\geq 1$ from a finite number of them, in order to illustrate the results described above. 

Note that the polynomial $f$ comes from the well-known Suzuki curve, defined over $\F_q$ by the equation
\begin{equation}
\label{suzu}
S_h:~y^q+y=x^{q_0}(x^q+x)
\end{equation}
The number of rational points of this curve over any extension of $\F_q$ is given in \cite[Proposition 4.3]{han}. Actually this curve is defined over the base field $\F_2$, and as an application of the preceding results, we give the number of its rational points over any field $\F_{2^n}$.

In order to do this, we determine the sums 
\[
S_n(f):=\sum_{x\in \F_{2^n}}(-1)^{\Tr_{\F_{2^n}/\F_2}(f(x))} 
\] 
for any $n\geq 1$.

First remark the following fact: since we have $x^{q_0+q}=(x^{1+2q_0})^{q_0}$, for any $x\in \F_{2^n}$, we have
\[
\Tr_{\F_{2^n}/\F_2}(f(x))=\Tr_{\F_{2^n}/\F_2}(x^{q_0+q}+x^{q_0+1})=\Tr_{\F_{2^n}/\F_2}(x^{2q_0+1}+x^{q_0+1})=\Tr_{\F_{2^n}/\F_2}(xR(x))
\]
where we have set $R(x):=x^{2q_0}+x^{q_0}$. With this additive polynomial, we obtain
\[
\widetilde{R}(x)=x^{2q}+x^q+x^2+x=(x^q+x)\circ(x^2+x)
\]
The roots of this polynomial form a $\F_2$-vector space of dimension $2h+2$ that contains $\F_q$ and $\F_4$. Since $2h+1$ is odd, this is the sub-vector space of $\F_{q^2}$ generated by $\F_q\cup\F_4$. We deduce easily the following

\begin{lemma}
\label{suzker}
Notations are as above. The decomposition field of $\widetilde{R}$ is $\F_{q^2}$, and we have for all $n\geq 1$
\[
c_n(f)=\left\{\begin{array}{rl} \pgcd(n,2h+1) & \textrm{if $n$ is odd} \\
\pgcd(n,2h+1)+1 & \textrm{if $n$ is odd} \\ \end{array}\right.
\]
\end{lemma}

We now evaluate some of the sums $S_n(f)$. First note than when $n$ divides $2h+1$, the field $\F_{2^n}$ is contained in $\F_q$, and we have $f(x)=0$ for all $x$ in $\F_{2^n}$. We deduce immediately the first assertion of the following

\begin{lemma}
\label{suzeval}
Notations are as above. We have 
\begin{itemize}
	\item[(1)] if $n$ divides $2h+1$, then $S_n(f)=2^n$;
	\item[(2)] if $n=4d$, where $d$ divides $2h+1$, then $S_n(f)=\chi(d)\chi(2h+1)2^{\frac{5d+1}{2}}$
\end{itemize}
\end{lemma}

\begin{proof}
We first choose $\alpha\in \F_4$ and $\beta\in \F_{16}$ such that $\alpha^2+\alpha=1$ and $\beta^2+\beta=\alpha$. Then $\{1,\alpha,\beta,\alpha\beta\}$ is a basis for the $\F_2$-vector space $\F_{16}$, and since $d$ is odd, it remains a basis for the $\F_{2^d}$-vector space $\F_{2^n}$.

Thus, for any $x\in \F_{2^n}$, we can write $x=x_0+\alpha x_1+\beta x_2+\alpha\beta x_3$ where $(x_0,x_1,x_2,x_3)\in \F_{2^d}^4$. After some calculations, we get
\[
\Tr_{\F_{2^n}/\F_{2^d}}(f(x))=x_1^{q_0}x_3+x_1x_3^{q_0}+x_2^{q_0+1}+x_2x_3^{q_0}+\epsilon x_3^{q_0+1}
\]
where $\epsilon=0$ if $\chi(2h+1)=1$ and $\epsilon=1$ if $\chi(2h+1)=-1$. Putting this into the sum, we get
\begin{eqnarray*}
S_n(f) & = & \sum_{(x_0,x_1,x_2,x_3)\in \F_{2^d}^4} \psi\left(\Tr_{\F_{2^d}/\F_{2}}(x_1^{q_0}x_3+x_1x_3^{q_0}+x_2^{q_0+1}+x_2x_3^{q_0}+\epsilon x_3^{q_0+1})\right) \\
 & = & 2^d \sum_{(x_2,x_3)\in \F_{2^d}^2} \psi\left(\Tr_{\F_{2^d}/\F_{2}}(x_2^{q_0+1}+x_2x_3^{q_0}+\epsilon x_3^{q_0+1})\right)S(x_3)\\
\end{eqnarray*}
where we have set $S(x_3)=\sum_{x_1\in \F_{2^d}} \psi\left(\Tr_{\F_{2^d}/\F_{2}}(x_1^{q_0}x_3+x_1x_3^{q_0})\right)$. Now since we have $\Tr_{\F_{2^d}/\F_{2}}(x_1^{q_0}x_3)=\Tr_{\F_{2^d}/\F_{2}}(x_1x_3^{2q_0})$, we deduce from an orthogonality relation that the sum $S(x_3)$ is zero, except when $x_3\in\F_2$, and then $S(x_3)=2^d$. We get 
\begin{eqnarray*}
S_n(f) & = & 2^{2d} \left(\sum_{x_2\in \F_{2^d}} \psi\left(\Tr_{\F_{2^d}/\F_{2}}(x_2^{q_0+1})\right)+\sum_{x_2\in \F_{2^d}} \psi\left(\Tr_{\F_{2^d}/\F_{2}}(x_2^{q_0+1}+x_2+\epsilon)\right)\right)\\
& = & 2^{2d} \left(\sum_{x_2\in \F_{2^d}} \psi\left(\Tr_{\F_{2^d}/\F_{2}}(x_2^{q_0+1})\right)+\chi(2h+1)\sum_{x_2\in \F_{2^d}} \psi\left(\Tr_{\F_{2^d}/\F_{2}}(x_2^{q_0+1}+x_2)\right)\right)\\
\end{eqnarray*}
The first sum is associated to the polynomial $xR_1(x)$, with $R_1(x)=x^{q_0}$. The roots of the polynomial $\widetilde{R}_1$ are the elements of the field $\F_{q_0^2}$, and since $d$ divides $2h+1$, the only roots in $\F_{2^d}$ are the elements in $\F_2$. Now since $d$ is odd the restriction of the quadratic form $\psi\left(\Tr_{\F_{2^d}/\F_{2}}(x_2^{q_0+1})\right)$ to $\F_2$ is non trivial, and the first sum is zero.

Finally, we apply \cite[Corollary 3]{lgw} to the second sum: it is equal to $\chi(d)2^{\frac{d+1}{2}}$, and this gives the desired result.
\end{proof}

With these results at hand, we are able to determine all sums $S_n(f)$. From Lemma \ref{suzker}, it is sufficient to give the invariants $\varepsilon_n(f)$

\begin{proposition}
\label{suzsum}
Recall that $f(x)=x^{q_0}(x^q+x)$. Then for all $n\geq 1$ we have
\begin{itemize}
	\item[(1)] if $n$ is odd, then 
	\[
	\varepsilon_n(f)=\chi\left(n\pgcd(n,2h+1)\right)
	\]
	\item[(2)] if $n$ is even, then 
	\[
	\varepsilon_n(f)=\left\{\begin{array}{rl} 0 & \textrm{if $v_2(n)=1$} \\
 \chi\left((2h+1)\pgcd(n,2h+1)\right) & \textrm{if $v_2(n)=2$} \\
 -\chi\left((2h+1)\pgcd(n,2h+1)\right) & \textrm{if $v_2(n)\geq 3$} \\ \end{array}\right.
	\]
\end{itemize}
\end{proposition}

\begin{proof}
From Lemma \ref{suzker}, the degree of the decomposition field of $\widetilde{R}$ over $\F_2$ is $N=4h+2$. From Lemma \ref{suzeval} (2), we have $\varepsilon_{2N}(f)=1$, and we deduce from Proposition \ref{percar2} that $D=4N$.

Now assertion (1) comes readily from Proposition \ref{moddnodd}.

From Lemma \ref{suzeval}, we have $\varepsilon_{4d}(f)=\chi\left((2h+1)d\right)$ for all divisors of $2h+1$. Now assertion (2) is a consequence of Proposition \ref{moddneven} since for all $n$ such that $v_2(n)\geq 2$, we have $\pgcd(n,2N)=4\pgcd(n,2h+1)$.
\end{proof}

We deduce some results on the factorization of the modified $L$-function below. They are an immediate consequence of the above result, and of Propositions \ref{multodd+} (iiib) and \ref{expmodd} (2).

\begin{proposition}
For $f$ as above, the only non zero multiplicities are among those $M_{8\ell}^\pm(f)$, $\ell|2h+1$.

Moreover, we have $M_{8\ell}^-(f)=\chi(\ell)\chi(2h+1)M_{8\ell}^+(f)$, and the $M_{8\ell}^+(f)$ are the solutions of the system
\[
A(2h+1)\left(M_{8\ell}^+(f)\right)_{\ell|2h+1}=\chi(2h+1)\left(\chi(d)2^{\frac{d-3}{2}}\right)_{d|2h+1}
\]
\end{proposition}

We end with the determination of the number of rational points of the curve $S_h$ over any extension of $\F_2$

\begin{proposition}
For any integer $n\geq 1$, we have
\[
\#S_h(\F_{2^n})=2^n+1+(2^{\pgcd(n,2h+1)}-1)S_n(f)
\]
\end{proposition}

\begin{proof}
First observe that the equation $y^q+y=t$ has $\sum_{z\in \F_{2^n}\cap\F_q} \psi\circ\Tr_{\F_{2^n}/\F_2} (tz)$ solutions for any $t\in \F_{2^n}$.

We deduce that the number of affine rational points over $\F_{2^n}$ of the curve $S_h$ is
\[
\sum_{z\in \F_{2^n}\cap\F_q} S_n(zf)
\]
When $z=0$, the sum is $2^n$. When $z\neq 0$, we remark that for any $t\in \F_q$, we have $f(tx)=t^{q_0+1}f(x)$. Since $q_0+1$ is prime to $q-1$, for any $z\in \F_q\cap\F_{2^n}$, there exists an unique $t\in \F_q\cap\F_{2^n}$ such that $t^{q_0+1}=z$. Thus we have $S_n(zf)=S_n(f)$, and this is the desired result
\end{proof}

\bibliographystyle{plain}
\bibliography{SommesQuad}

\end{document}